\documentclass[a4 paper,12pt,bezier]{article}
\usepackage[top=3cm,bottom=3cm,left=2.5cm,right=2.5cm]{geometry}
\usepackage{amsmath}
\usepackage{amsfonts,amsthm,amssymb}
\usepackage{mathrsfs}
\usepackage{amsfonts}
\usepackage{graphics,subfigure,caption2}
\usepackage{tikz, color}
\usepackage[colorlinks=true,anchorcolor=blue,filecolor=blue,linkcolor=blue,urlcolor=blue,citecolor=blue]{hyperref}
\usepackage{graphics,epstopdf}
\usepackage{latexsym,bm}
\usepackage{amsfonts,amsthm,amssymb,mathrsfs,bbding}
\usepackage [latin1]{inputenc}
\usepackage{titlesec}
\usepackage[titletoc]{appendix}
\usepackage{indentfirst}
\usepackage{listings}
\usepackage{float}
\newtheorem{thm}{Theorem}

\newtheorem{lem}{Lemma}
\newtheorem{false statement}{False statement}

\theoremstyle{definition}

\newtheorem{conj}{Conjecture}

\newtheorem{problem}{Problem}

\begin{document}

\title{Connectivity preserving spanning $(u,v)$-paths in $k$-connected graphs\footnote{Supported by Natural Science Foundation of Xinjiang Uygur Autonomous Region (2025D01E02), National Natural Science Foundation of China (12261086), Natural Science Foundation of Xinjiang Uygur Autonomous Region (2024D01C41) and Tianshan Talent Training Program (2024TSYCCX0013).}}
\author{{Zhaolin Teng},  {Yingzhi Tian}\thanks{Corresponding author. E-mail: zlteng03@163.com (Z. Teng), tianyzhxj@163.com (Y. Tian).}\\
{\footnotesize College of Mathematics and System Sciences, Xinjiang
	University, Urumqi, Xinjiang, 830046, PR China}}
\date{}

\maketitle {\flushleft\large\bf Abstract:}
Hasunuma [Graphs Combin. 41:10 (2025)] proved that for $k\ge 2$, there exists a function $f(k)=O(k)$ such that every $k$-connected graph $G$ of order $n\ge f(k)$ with $\delta(G)\ge \frac{n}{2}$ contains a Hamiltonian cycle $H$ such that $G-E(H)$ is $k$-connected. 
In this paper, we show that for $k\ge 2$, if $G$ is a $k$-connected graph of order $n\ge 6k+6$ with minimum degree at least $\frac{n+1}{2}$, then for any two distinct vertices $u,v\in V(G)$, there exists a Hamiltonian $(u,v)$-path $P$ such that $G-E(P)$ is $k$-connected. Moreover, we further extend this result to $s$ internally disjoint spanning $(u,v)$-paths. 

\vspace{0.1cm}
\begin{flushleft}
\textbf{Keywords:} Connectivity; $k$-connected graph; Hamiltonian path; Spanning path
\end{flushleft}

\section{Introduction}

Throughout this paper, all graphs are finite and without multiple edges and without loops. For graph theoretical terminologies and notation not defined here, we follow \cite{Bondy}. Let $G$ be a graph with vertex set $V(G)$ and edge set $E(G)$. The order of $G$ is the cardinality of its vertex set, denoted by $|V(G)|$. 
For $v\in V(G)$, let $N_G(v)$ be the set of vertices adjacent to $v$ in $G$, i.e., $N_G(v)=\{u\in V(G)| uv\in E(G)\}$, and let $d_G(v)=|N_G(v)|$ be the degree of $v$ in $G$. Let $\delta(G)=\min_{v\in V(G)}d_G(v)$ and $\Delta(G)=\max_{v\in V(G)}d_G(v)$. For a nonempty vertex subset $S$ of a graph $G$, we denote by $G[S]$ the subgraph induced by $S$ and by $G-S$ the subgraph induced by $V(G)-S$. For a nonempty edge subset $F\subseteq E(G)$, $G-F$ is defined similarly. The connectivity of $G$, denoted by $\kappa(G)$, is the minimum size of a vertex set $S\subseteq V(G)$ such that $G-S$ is disconnected or has only one vertex. The edge connectivity of $G$, denoted by $\lambda(G)$, is the minimum size of an edge set $F\subseteq E(G)$ such that $G-F$ is disconnected. The graph $G$ is said to be \emph{$k$-connected} (respectively, \emph{$k$-edge-connected}) if $\kappa(G)\geq k$ (respectively, $\lambda(G)\geq k$). By the definition of connectivity, we know that a graph is $k$-connected if it remains connected after deletion of any fewer than $k$ vertices. A graph of order $1$ is called trivial. Note that we consider the trivial graph is both $1$-connected and $1$-edge-connected.

In 1972, Chartrand, Kaigars and Lick proved the following well-known result.

\begin{thm}(\cite{Chartrand})
    Every $k$-connected graph $G$ of minimum degree $\delta(G)\geq \lfloor \frac{3k}{2}\rfloor$ has a vertex $u$ such that $\kappa(G-u)\ge k$.
\end{thm}

In 2008, Fujita and Kawarabayashi in \cite{Fujita} proved that there are two redundant adjacent vertices in a $k$-connected graph with minimum degree at least $\lfloor \frac{3k}{2}\rfloor+2$. In the same paper, they stated the following conjecture.

\begin{conj}\label{conj1}(\cite{Fujita})
    For all positive integers $k,m$, there is a (least) non-negative integer $f_k(m)$ such that every $k$-connected graph $G$ with $\delta(G)\geq \lfloor \frac{3k}{2}\rfloor-1+f_k(m)$ contains a connected subgraph $W$ of exact order $m$ such that $G-V(W)$ is still $k$-connected.
\end{conj}

Mader \cite{Mader} confirmed Conjecture \ref{conj1} and proved that $f_k(m)=m$ holds for all $k,m$. In addition, the connected subgraph $W$ can even be a path. Meanwhile, Mader conjectured that the result would hold for all trees. 

\begin{conj}\label{conj2}(\cite{Mader})
    For any tree $T$ with order $m$, every $k$-connected graph $G$ with $\delta(G)\geq \lfloor \frac{3k}{2} \rfloor +m-1$ contains a tree $T'\cong T$ such that $\kappa(G-V(T'))\geq k$.
\end{conj}

Actually, the main result in \cite{Diwan} implies that conjecture \ref{conj2} holds for $k=1$. For $k=2$, many researchers made contributions to solve Conjecture \ref{conj2}, see \cite{Hasunuma1,Hasunuma3,Hong, Lu, Tian, Tian2}. In 2022, Hong and Liu \cite{Hong1} confirmed Conjecture \ref{conj2} for $k=2$ and $k=3$. As far as we know, Conjecture \ref{conj2} remains open for $k\geq 4$. 

As a similar problem on the edge-version, we can consider the graph obtained from $G$ by deleting all edges of a subtree $T'$  instead of the vertices of $T'$. Hasunuma \cite{Hasunuma} proposed the following edge-version conjecture.

\begin{conj}\label{conj4}(\cite{Hasunuma})
    For any tree $T$ of order $m$, every $k$-connected graph (respectively, $k$-edge-connected graph) $G$ with minimum degree at least $k+m-1$ contains a tree $T'\cong T$ such that $\kappa(G-E(T'))\geq k$ (respectively, $\lambda (G-E(T'))\geq k$).
\end{conj}

Hasunuma confirmed Conjecture~\ref{conj4} for $k\le 2$ in \cite{Hasunuma}. Later, authors in \cite{H.Liu} and \cite{Yang2} independently verified Conjeture~\ref{conj4} for $k=3$. For generally $k\ge 4$, Conjecture~\ref{conj4} remains open. 

A path or cycle which contains every vertex of a graph is called a \emph{Hamilton path} or \emph{Hamilton cycle} of the graph. A graph is \emph{Hamiltonian} if it contains a Hamilton cycle. Dirac's theorem asserts that every graph $G$ of order $n$ with $\delta(G)\ge \frac{n}{2}$ is Hamiltonian, and the graph with this degree condition is often called a Dirac graph. A natural problem asks whether one can find a Hamiltonian structure whose removal preserves the connectivity. Recently, Hasunuma \cite{Hasunuma4} considered connectivity preserving Hamiltonian cycles in Dirac graphs.

\begin{thm}(\cite{Hasunuma4})\label{thm:ham-cycle}
    For any $k\ge 2$, every $k$-connected graph $G$ of order $n\ge 6k+10$ with $\delta(G)\ge \frac{n}{2}$ contains a Hamiltonian cycle $H$ such that $G-E(H)$ is $k$-connected.
\end{thm}

\begin{thm}(\cite{Hasunuma4})\label{thm:ham-cycles}
    For any $k\ge 2$ and any $l\ge 2$, every $k$-connected graph $G$ of order
    $$n\ge \max\{kl+\max\{kl,6l+2\}+3k+2l-6, 6k+20l-10,\frac{224l}{5}-10\}$$ 
    with $\delta(G)\ge \frac{n}{2}$ contains $l$ Hamiltonian cycles $H_1, H_2,\ldots, H_l$ such that $G-\bigcup_{1\le i\le l}E(H_i)$ is $k$-connected.
\end{thm}

The \emph{spanning connectivity} of $G$, denoted by $\kappa^*(G)$, is the largest integer $t$ such that, for every $s$ with $1\le s\le t$ and every pair of distinct vertices $u,v\in V(G)$, there exist $s$ internally disjoint $(u,v)$-paths whose union spans $G$. This collection of $t$ internally disjoint $(u,v)$-paths is called a spanning $(u,v)$-path system.
In this terminology, the case $t=1$ is precisely the existence of a Hamiltonian $(u,v)$-path. The case $t=2$ is closely related to Hamiltonian cycles. Hence, Theorem~\ref{thm:ham-cycle} can be considered as the connectivity-preserving result of a spanning $(u,v)$-path system for $t=2$. Motivated by this direction, we establish an endpoint-prescribed Hamiltonian path version corresponding to $t=1$. Furthermore, under suitable minimum degree and order assumptions, we show that for every $s$ with $1\le s\le t$ and every pair of distinct vertices $u,v$, there exist $s$ internally disjoint spanning $(u,v)$-paths whose removal preserves $k$-connectivity. 

The main results of this paper are presented as follows.

\begin{thm}\label{thm:ham-path}
For any $k\ge 2$, let $G$ be a $k$-connected graph of order $n\geq 6k+6$ with $\delta(G)\ge \lceil \frac{n+1}{2} \rceil$. Then for any two distinct vertices $u,v\in V(G)$, there exists a Hamiltonian $(u,v)$-path $P$ such that $G-E(P)$ is $k$-connected.
\end{thm}

\begin{thm}\label{thm:path-system}
For any $k\ge 2$ and $t\ge 3$, let $G$ be a $k$-connected graph of order $n$. Suppose that one of the following holds:
\begin{enumerate}
    \item[(i)] $n\ge 6k+1$ and $\delta(G)\ge \lceil \frac{n+6}{2}\rceil$ if $t=3$;
    \item[(ii)] $n\ge 6k+7t-17$ and $\delta(G)\ge \lceil \frac{n+t+2}{2}\rceil$ if $t\ge 4$.
\end{enumerate}
Then for any two distinct vertices $u,v\in V(G)$ and for any positive integer $s$ with $1\le s\le t$, there exist $s$ internally vertex-disjoint $(u,v)$-paths $P_1,P_2,\ldots, P_s$ such that $V(P_1\cup P_2\cup\cdots \cup P_s)=V(G)$ and $\kappa(G-E(P_1\cup P_2\cup \cdots \cup P_s))\ge k$.
\end{thm}

The paper is organized as follows. Section~\ref{sec2} presents structural properties of graphs concerning connectivity and spanning connectivity. The proofs of our two main results are given in Sections~\ref{sec3} and~\ref{sec4}, respectively. Section~\ref{sec5} concludes the paper.

\vspace*{2mm}

\section{Preliminaries}\label{sec2}

In this section, we collect several tools that will be used repeatedly in the proofs of the main theorems.
We first recall the definition of spanning connectivity. For two distinct vertices $u,v$ of a graph $G$, a collection of internally disjoint $(u,v)$-paths is called a spanning $(u,v)$-path system if the union of their vertex sets is $V(G)$. The \emph{spanning connectivity} $\kappa^*(G)$ is the largest integer $t$ such that, for every integer $s$ with $1\le s\le t$ and every pair of distinct vertices $u,v\in V(G)$, there exist $s$ internally disjoint spanning $(u,v)$-paths in $G$. The following lower bound due to Lin, Huang and Hsu will be used to find the initial spanning path system.

\begin{lem}(\cite{Lin})\label{lem7}
    If $G$ is a graph with $\frac{n}{2}+1\leq \delta(G)\leq n-2$, then $\kappa^*(G)\geq 2\delta(G)-n+2$.
\end{lem} 

A graph is \emph{Hamiltonian-connected} if every two distinct vertices are joined by a Hamiltonian path. For a graph $G$ and nonnegative integer $j$, let $\psi_j (G)=|\{v\in V(G)| d_G(v)\le j\}|$. In 1969, Chartrand, Kapoor and Kronk gave the following sufficient condition for a graph to be Hamiltonian-connected. 

\begin{lem}(\cite{Chartrand1})\label{lem:hconn}
Let $G$ be a graph of order $n\ge 4$ such that for every integer $j$ with $2\le j\le \frac{n}{2}$. If the number of vertices of degree not exceeding $j$ is less than $j-1$, that is, $\psi_j (G)<j-1$, then $G$ is Hamiltonian-connected.
\end{lem}

In 1968, Chartrand and Harary showed the following sufficient condition for a graph to be $k$-connected.

\begin{lem}(\cite{Chartrand2})\label{lem:CH}
For any $k\ge 1$, every graph of order $n\ge k+1$ with $\delta(G)\ge \frac{n+k-2}{2}$ is $k$-connected.
\end{lem}

A vertex of degree $2$ in a path $P$ is called an internal vertex in $P$ and the set of internal vertices in $P$ is denoted by $V_I(P)$. The following lemma from Hasunuma gives a construction of a $k$-connected graph from two vertex-disjoint $k$-connected graphs. 

\begin{lem}(\cite{Hasunuma4})\label{lem:two-block}
Let $G_1$ and $G_2$ be two $k$-connected graphs such that $V(G_1)\cap V(G_2)=\emptyset$. Let $G$ be a graph obtained from $G_1$ and $G_2$ by adding $k$ vertex-disjoint paths $Q_1, Q_2, \ldots,Q_k$ connecting vertices in $V(G_1)$ and $V(G_2)$, where $(V(G_1)\cup V(G_2))\cap (\bigcup_{1\le i \le k}V_I(Q_i))=\emptyset$. Let $G'$ be a graph obtained from $G$ by adding edges so that for each $1\le i \le k$, every internal vertex in $Q_i$ has $k$ neighbors in $V(G_1)\cup V(G_2)$. Then $G'$ is $k$-connected.
\end{lem}

The final lemma explains that once a $k$-connected core has been obtained, every remaining vertex can be added back provided it has at least $k$ neighbors in the core.

\begin{lem}\label{lem:add vertex}
    Let $G$ be a $k$-connected graph, and let $x\notin V(G)$. Let $G'$ be a graph obtained from $G$ by adding the vertex $x$ and joining $x$ to at least $k$ vertices of $G$. Then $G'$ is $k$-connected.
\end{lem}

\begin{proof}
    Take any $S\subseteq V(G')$ with $|S|\le k-1$. We show that $G'-S$ is connected. First suppose that $x\in S$. Then $G'-S=G-(S\setminus \{x\})$. Since $|S\setminus \{x\}|\le k-2<k$ and $G$ is $k$-connected, $G'-S$ is connected. Now suppose that $x\notin S$. Since $|S|\le k-1$ and $S\subseteq V(G)$, the graph $G-S$ is connected. Then $G'-S$ contains $G-S$ as a connected subgraph. Moreover, since $|N_{G'}(x)|\ge k$ and $S$ has size at most $k-1$, the vertex $x$ has at least one neighbor in $G-S$. Thus, $G'-S$ is connected. Therefore, $G'$ is $k$-connected. 
\end{proof}

\section{Proof of Theorem~\ref{thm:ham-path}}\label{sec3}

We prove Theorem~\ref{thm:ham-path} in this section.

\begin{proof}[Proof of Theorem~\ref{thm:ham-path}]
Let $u,v\in V(G)$ be any two distinct vertices. Since $\delta(G)\ge \lceil \frac{n+1}{2}\rceil$, Lemma~\ref{lem:hconn} implies that $G$ is Hamiltonian-connected. Hence there exists a Hamiltonian $(u,v)$-path, denoted by $P_0$, in $G$. If $\kappa(G-E(P_0))\ge k$, then we are done by letting $P=P_0$. Thus we assume that $G-E(P_0)$ is not $k$-connected. Let $G'=G-E(P_0)$. Also, let $k'=\kappa(G')<k$ and $W=\{w_1,w_2,\ldots, w_{k'}\}\subseteq V(G')$
such that $G'-W$ is disconnected.

\vspace{0.1cm}
\noindent \textbf{Claim 1.}
    The graph $G'-W$ has exactly two connected components.

\vspace{0.1cm}
\noindent \textbf{Proof of Claim 1.} Since $\Delta(P_0)\le 2$, we have $\delta(G')\ge \delta(G)-2$. Let $G_1, G_2, \ldots, G_r$ be the connected components of $G'-W$ such that $|V(G_1)|\le |V(G_2)|\le \cdots \le |V(G_r)|$. For any $x\in V(G_i)$, where $1\le i \le r$, $x$ has no neighbor in $V(G)\setminus (V(G_i)\cup W)$. Thus we obtain for each $i$ with $1\le i \le r$,
$$d_{G'}(x)\le |V(G_i)|-1+|W|,$$
and hence
$$|V(G_i)|\ge d_{G'}(x)-|W|+1\ge \delta(G')-k'+1\ge \lceil\frac{n+1}{2}\rceil -2-k'+1.$$
Therefore,
$$|V(G_i)|\ge \lceil\frac{n+1}{2}\rceil-k'-1=\lfloor\frac{n}{2}\rfloor-k'.$$
Suppose that $r\ge 3$. Then we have $|V(G_1)|\le \frac{n-k'}{3}$. Thus,
$$\lfloor\frac{n}{2}\rfloor-k'\le \frac{n-k'}{3},$$
that is
$$n\le 4k'+3\le 4k-1,$$
which contradicts the assumption that $n\ge 6k+6$. Thus $r=2$.

Let $G_1$ and $G_2$ be the two components of $G'-W$. Since $|V(G_1)|+|V(G_2)|=n-k'$ and $|V(G_i)|\ge \lfloor\frac{n}{2}\rfloor-k'$, we have $|V(G_i)|\le \lceil\frac{n}{2}\rceil$ for $i=1,2$.

\vspace{0.1cm}
\noindent \textbf{Claim 2.}
    There are $k$ vertex-disjoint paths $Q_1, Q_2,\ldots, Q_k$ in $G$ connecting vertices in $V(G_1)$ and vertices in $V(G_2)$ such that $2\le |V(Q_i)|\le 4$ for each $i$ with $1\le i \le k$ and $\bigcup_i^kV_I(Q_i)\subseteq W$.

\vspace{0.1cm}
\noindent \textbf{Proof of Claim 2.} We first know that $V(G)=V(G_1)\cup W\cup V(G_2)$ is a disjoint union. For any $w\in W$, we have
$$ \begin{aligned}
d_G(w) &=|N_G(w)\cap V(G_1)|+|N_G(w)\cap V(G_2)|+|N_G(w)\cap W| \\
 &\le |N_G(w)\cap V(G_1)|+|N_G(w)\cap V(G_2)|+(|W|-1),
\end{aligned}
$$
and hence
$$
|N_G(w)\cap V(G_1)|+|N_G(w)\cap V(G_2)| \ge d_G(w)-(|W|-1)\ge \delta(G)-k'+1.
$$
Since $\delta(G)\ge \lceil\frac{n+1}{2}\rceil$, $n\ge 6k+6$ and $k'\le k-1$, we get 
$$
|N_G(w)\cap V(G_1)|+|N_G(w)\cap V(G_2)| \ge3k+4-(k-1)+1 =2k+6.
$$
Thus, each $w\in W$ has at least $k+1$ neighbors in $G_1$ or at least $k+1$ neighbors in $G_2$.

By the $k$-connectivity of $G$, Menger's theorem gives $k$ pairwise vertex-disjoint paths $P_1, P_2,\ldots, P_k$ connecting $G_1$ and $G_2$. If $P_i$ has no vertex in $W$, then $P_i$ has a subpath $P_i'=(u_i,v_i)$ of order $2$, where $u_i\in V(G_1)$ and $v_i\in V(G_2)$. If $P_i$ has a vertex in $W$, then $P_i$ has either a subpath $P_i'=(u_i, v_i)$ or $P_i'=(u_i,w_{i,1},w_{i,2},\ldots, w_{i,r_i},v_i)$, where $u_i\in V(G_1), w_{i,s}\in W$ and $v_i\in V(G_2)$. Note that these subpaths are still vertex-disjoint. Thus, without loss of generality, we may assume that $|V(P_i')|=2$ for each $1\le i\le k-l$ and $|V(P_i')|\ge 3$ for each $k-l+1\le i\le k$. Let $\Pi =\{P_1', P_2',\ldots, P_k'\}$. Also let $U=\{u_1, u_2,\ldots, u_k\}$ and $V=\{v_1,v_2,\ldots,v_k\}$.

Consider a path $P_i'=(u_i,w_{i,1},w_{i,2},\ldots, w_{i,r_i},v_i)$, where $r_i\ge 2$, that is, $|V(P_i')|\ge 4$. If $|N_G(w_{i,2})\cap V(G_1)|\ge k+1$, then we can select a vertex $u_i'\in (N_G(w_{i,2})\cap V(G_1))\setminus U$ and obtain a path $P_i''=(u_i',w_{i,2},\ldots, w_{i,r_i},v_i)$ which is vertex-disjoint with any path in $\Pi \setminus \{P_i'\}$. If $|N_G(w_{i,2})\cap V(G_2)|\ge k+1$, then we can select a vertex $v_i'\in (N_G(w_{i,2})\cap V(G_2))\setminus V$ and obtain a path $P_i''=(u_i, w_{i,1},w_{i,2},v_i')$. After finite times of replacement, we may assume, without loss of generality, that $3\le|V(P_i')|\le 4$ for any $k-l+1\le i\le k$ and $|V(P_i')|=2$ for any $1\le i \le l$. Moreover, we have $\bigcup_i^k V_I(P_i')\subseteq W$. Therefore, Claim 2 holds.

Let $S=\bigcup_{i=1}^k V_I(Q_i)$. Then $S\subseteq W$. Let $B=G[\bigcup_{i=1}^{k}E(Q_i)]$. Since $Q_1, Q_2,\ldots, Q_k$ are pairwise vertex-disjoint, we have $\Delta(B)\le 2$ and $|V(B)|=2k+|S|\le 2k+|W|=2k+k'\le 3k-1$.

\vspace{0.1cm}
\noindent \textbf{Claim 3.}
    $G-E(B)$ is Hamiltonian-connected.

\vspace{0.1cm}

\noindent \textbf{Proof of Claim 3.} Let $G''=G-E(B)$. For any vertex $x\in V(G'')$, if $x\notin V(B)$, then $d_{G''}(x)=d_G(x)\ge \lceil\frac{n+1}{2}\rceil$; if $x\in V(B)$, then $d_{G''}(x)\ge \lceil\frac{n+1}{2}\rceil-2$. Hence, the vertices of $G''$ with degree not exceeding $\frac{n}{2}$ are all contained in $V(B)$. Thus, for $2\le j\le \frac{n}{2}$, we have
$$\psi_j(G'')\le |V(B)|\le 3k-1.$$
More accurately, we have $\psi_j(G'')=0$ with $j<\lceil\frac{n+1}{2}\rceil-2$ and $\psi_j(G'')\le 3k-1<j-1$ with $\lceil\frac{n+1}{2}\rceil-2\le j\le \frac{n}{2}$, since $n\ge 6k+6$. By Lemma~\ref{lem:hconn}, $G''$ is Hamiltonian-connected. 

Therefore, for any $u,v\in V(G'')$, there exists a Hamiltonian path $P$ connecting $u$ and $v$. Since $E(P)\subseteq G-E(B)$, we have $E(B)\subseteq G-E(P)$. 

\vspace{0.1cm}
\noindent \textbf{Claim 4.}
    $G-E(P)$ is $k$-connected.

\vspace{0.1cm}
\noindent \textbf{Proof of Claim 4.} Let $G^*=G-E(P)$. Since $E(B)\subseteq G^*$, $G^*$ contains $k$ pairwise vertex-disjoint paths $Q_1, Q_2,\ldots, Q_k$. For $i=1,2$, let $G_i^*=G^*[V(G_i)]$. We claim that $G_i^*$ is $k$-connected. For any $x\in V(G_i)$, since $G_i$ is a component of $G'-W=(G-E(P_0))-W$, the edges incident with $x$ to other component in $G$ come only from $P_0$. Thus,
$$|N_G(x)\cap V(G_i)|\ge \delta(G)-2-k'.$$
After removing $E(P)$, we have
$$d_{G_i^*}(x)\ge \delta(G)-2-k'-2\ge \lceil\frac{n+1}{2}\rceil -k-3.$$
Since $|V(G_i)|\le \lceil \frac{n}{2}\rceil$ and $n\ge 6k+6$, it is easy to verify that 
$$\lceil\frac{n+1}{2}\rceil -k-3 \ge \frac{|V(G_i)|+k-2}{2},$$
and hence $\delta(G_i^*)\ge \frac{|V(G_i)|+k-2}{2}$. Therefore $G_i^*$ is $k$-connected by Lemma~\ref{lem:CH}.

Let $F=G^*[V(G_1)\cup S \cup V(G_2)]$. We know that $F$ contains the two $k$-connected graphs $G^*[V(G_1)]$ and $G^*[V(G_2)]$, together with the $k$ pairwise internally disjoint paths $Q_1, \ldots, Q_k$ connecting $V(G_1)$ and $V(G_2)$. For any $w\in S$, we have
$$
|N_G(w)\cap (V(G_1)\cup V(G_2))| \ge \delta(G)-(|W|-1) \ge \delta(G)-k'+1.
$$
After deleting $E(P)$, we have
$$
|N_{G^*}(w)\cap (V(G_1)\cup V(G_2))| \ge \delta(G)-k'+1-2 \ge \left\lceil\frac{n+1}{2}\right\rceil-k \ge k.
$$
Hence, each $w\in S$ has at least $k$ neighbors belonging to $V(G_1)\cup V(G_2)$ in $G^*$. It follows from Lemma~\ref{lem:two-block} that $F$ is $k$-connected.

Let $W_0=W\setminus S$. If $W_0=\emptyset$, then we are done. Let $w\in W_0$, the same estimate as above gives
$$
|N_{G^*}(w)\cap (V(G_1)\cup V(G_2))| \ge \delta(G)-k'+1-2 \ge \left\lceil\frac{n+1}{2}\right\rceil-k \ge k.
$$
Since $V(G_1)\cup V(G_2)\subseteq V(F)$, the vertex $w$ has at least $k$ neighbors in $F$. By Lemma \ref{lem:add vertex}, $G^*[V(F)\cup \{w\}]$ is $k$-connected. Repeat this process for all vertices in $W_0$. Then we conclude that $G^*[V(F)\cup W_0]$ is $k$-connected. Therefore, the graph $G^*$ is $k$-connected. 
That is, $\kappa(G-E(P))\ge k$. The theorem follows.
\end{proof}

\section{Proof of Theorem~\ref{thm:path-system}}\label{sec4}

We give the proof of Theorem~\ref{thm:path-system} in this section.

\begin{proof}[Proof of Theorem~\ref{thm:path-system}]
For any $u,v\in V(G)$ and an integer $s$ with $1\le s\le t$, we have $\max\{s,2\}\le t$ since $t\ge 3$. We divide the discussion into two cases.

\noindent \textbf{Case 1.} $\delta(G)=n-1$.

In this case, $G=K_n$. We claim that $n-2\ge s$. If $t=3$, then $n\ge 6k+1\ge 13$, and hence $n-2\ge 11\ge s$. If $t\ge 4$, then $n\ge 6k+7t-17\ge 7t-5$, and hence $n-2\ge 7t-7\ge t\ge s$. Thus, $V(G)\setminus \{u,v\}$ results in $s$ nonempty sets $X_1,X_2,\ldots, X_s$, and there is a path from $u$ to $v$ with $V_I(P_i)=X_i$ for each $i=1,2,\ldots, s$ as $G=K_n$. Then $P_1,P_2,\ldots, P_s$ are $s$ internally vertex-disjoint $(u,v)$-paths. Let $\mathcal{P}=P_1\cup P_2\cup \cdots \cup P_s$. We have $V(\mathcal{P})=V(G)$ and $\Delta(\mathcal{P})=\max\{s,2\}\le t$. So,
$$\delta(G-E(\mathcal{P}))\ge n-1-t.$$
To prove that $G-E(P)$ is $k$-connected, it is enough, by Lemma~\ref{lem:CH}, to verify 
$$n-1-t\ge \frac{n+k-2}{2},$$
which is equivalent to $n\ge k+2t$. If $t=3$, then 
$$n-(k+2t)\ge (6k+1)-(k+6)=5k-5>0.$$ 
If $t\ge 4$, then 
$$n-(k+2t)\ge (6k+7t-17)-(k+2t)=5k+5t-17>0.$$
Thus $n\ge k+2t$ in both cases. Therefore, $G-E(\mathcal{P})$ is $k$-connected.

\noindent \textbf{Case 2.} $\delta(G)\le n-2$.

By the minimum degree condition, we have $\delta(G)\ge \frac{n}{2}+1$. Lemma~\ref{lem7} implies that 
$$\kappa^*(G)\ge 2\delta(G)-n+2.$$
If $t=3$, then
$$2\delta(G)-n+2\ge 2\lceil \frac{n+6}{2}\rceil -n+2\ge 8\ge t.$$
If $t\ge 4$, then
$$2\delta(G)-n+2\ge 2\lceil \frac{n+t+2}{2}\rceil -n+2\ge t+4\ge t.$$
Hence, there are $s$ internally disjoint $(u,v)$-paths $P_1,P_2,\ldots, P_s$ satisfying $V(P_1\cup P_2\cup \cdots \cup P_s)=V(G)$, where $1\le s\le t$. Let $\mathcal{P}_0=P_1\cup P_2\cup \cdots \cup P_s$. We have  $\Delta(\mathcal{P}_0)=\max\{s,2\}\le t$. If $\kappa(G-E(\mathcal{P}_0))\ge k$, then we are done. Assume that $\kappa(G-E(\mathcal{P}_0))< k$. Let $G'=G-E(\mathcal{P}_0)$. Also, let $k'=\kappa(G')<k$ and $W=\{w_1,w_2,\ldots, w_{k'}\}\subseteq V(G')$
such that $G'-W$ is disconnected.

\vspace{0.1cm}
\noindent \textbf{Claim 1.}
    The graph $G'-W$ has exactly two connected components. 

\vspace{0.1cm}   
\noindent \textbf{Proof of Claim 1.} Since $\Delta(\mathcal{P}_0)\le t$, we have $\delta(G')\ge \delta(G)-t$. Let $G_1, G_2, \ldots, G_r$ be the connected components of $G'-W$ such that $|V(G_1)|\le |V(G_2)|\le \cdots \le |V(G_r)|$. For any $x\in V(G_i)$, where $1\le i \le r$, $x$ has no neighbor in $V(G)\setminus (V(G_i)\cup W)$. Thus we obtain for each $i$ with $1\le i \le r$,
$$d_{G'}(x)\le |V(G_i)|-1+|W|,$$
and hence
$$
|V(G_i)| \ge d_{G'}(x)-|W|+1\ge \delta(G')-k'+1\ge \delta(G) -t-k'+1.
$$
If $t=3$, then
$$|V(G_i)|\ge \lceil \frac{n+6}{2}\rceil-3-k'+1=\frac{n}{2}-k'+1\ge \frac{n}{2}-\frac{t}{2}-k'+2.$$
If $t\ge 4$, then
$$|V(G_i)|\ge \lceil \frac{n+t+2}{2}\rceil-t-k'+1=\frac{n}{2}-\frac{t}{2}-k'+2.$$
Suppose that $r\ge 3$. Then we have $|V(G_1)|\le \frac{n-k'}{3}$. Thus,
$$\frac{n}{2}-\frac{t}{2}-k'+2\le \frac{n-k'}{3},$$
that is
$$n\le 3t+4k'-12\le 3t+4k-16,$$
which contradicts the assumption that $n\ge 6k+1$ for $t=3$ and that $n\ge  6k+7t-17$ for $t\ge 4$. Thus $r=2$.

Let $G_1$ and $G_2$ be the two components of $G'-W$. Since $|V(G_1)|+|V(G_2)|=n-k'$ and $|V(G_i)|\ge \delta(G)-t-k'+1$, we have $|V(G_i)|\le n-\delta(G)+t-1$ for $i=1,2$.

\vspace{0.1cm}
\noindent \textbf{Claim 2.}
    There are $k$ vertex-disjoint paths $Q_1, Q_2,\ldots, Q_k$ in $G$ connecting vertices in $V(G_1)$ and vertices in $V(G_2)$ such that $2\le |V(Q_i)|\le 4$ for each $i$ with $1\le i \le k$ and $\bigcup_i^kV_I(Q_i)\subseteq W$.

\vspace{0.1cm}
\noindent \textbf{Proof of Claim 2.} We know that $V(G)=V(G_1)\cup W\cup V(G_2)$ is a disjoint union. For any $w\in W$, we have
$$ \begin{aligned}
    d_G(w) &=|N_G(w)\cap V(G_1)|+|N_G(w)\cap V(G_2)|+|N_G(w)\cap W| \\
    & \le |N_G(w)\cap V(G_1)|+|N_G(w)\cap V(G_2)|+(|W|-1),
\end{aligned}
$$
and hence
$$ \begin{aligned}
|N_G(w)\cap V(G_1)|+|N_G(w)\cap V(G_2)| &\ge d_G(w)-(|W|-1)\\ 
&\ge \delta(G)-k'+1 \\
&\ge \delta(G)-k+2.
\end{aligned}
$$
If $t=3$, then
$$
\begin{aligned}
|N_G(w)\cap V(G_1)|+|N_G(w)\cap V(G_2)|& \ge \lceil \frac{n+6}{2} \rceil-k+2\\
&\ge \lceil \frac{6k+7}{2}\rceil -k+2\\
&\ge 2k+6.
\end{aligned}
$$
If $t\ge 4$, then
$$
\begin{aligned}
|N_G(w)\cap V(G_1)|+|N_G(w)\cap V(G_2)|& \ge \lceil \frac{n+t+2}{2} \rceil-k+2\\
&\ge \lceil \frac{6k+8t-15}{2}\rceil -k+2\\
&\ge 2k+11.
\end{aligned}
$$
Thus, each $w\in W$ has at least $k+1$ neighbors in $V(G_1)$ or at least $k+1$ neighbors in $V(G_2)$. The remaining proof process is the same as that of Claim~2 in Theorem~\ref{thm:ham-path}. Hence, Claim 2 holds.

Let $S=\bigcup_{i=1}^k V_I(Q_i)$. Then $S\subseteq W$. Let $B=G[\bigcup_{i=1}^{k}E(Q_i)]$. Since $Q_1, Q_2,\ldots, Q_k$ are pairwise vertex-disjoint, we have $\Delta(B)\le 2$ and $|V(B)|=2k+|S|\le 2k+|W|=2k+k'\le 3k-1$.

\vspace{0.1cm}
\noindent \textbf{Claim 3.}
    There are $s$ internally disjoint $(u,v)$-paths $P_1,P_2,\ldots, P_s$ in $G-E(B)$ such that $V(P_1\cup P_2\cup \cdots \cup P_s)=V(G)$, where $1\le s\le t$.

\vspace{0.1cm}

\noindent \textbf{Proof of Claim 3.} Let $G''=G-E(B)$. Since $\Delta(B)\le 2$, we have $\delta(G'')\ge \delta(G)-2$. The minimum degree condition guarantees $\delta(G'')\ge \frac{n}{2}+1$.
Thus, Lemma~\ref{lem7} applies to $G''$. Moreover,
$$2\delta(G'')-n+2\ge 2(\delta(G)-2)-n+2=2\delta(G)-n-2.$$
If $t=3$, then
$$2\delta(G'')-n+2\ge 2\lceil \frac{n+6}{2}\rceil -n-2\ge 4\ge t.$$
If $t\ge 4$, then
$$2\delta(G'')-n+2\ge 2\lceil \frac{n+t+2}{2}\rceil -n-2\ge t.$$
Therefore, $\kappa^*(G'')\ge t$, and $G''$ contains $s$ internally disjoint $(u,v)$-paths $P_1,\ldots,P_s$ whose union spans $G$ for any $s$ with $1\le s\le t$.

Let $\mathcal{P}=P_1\cup \cdots \cup P_s$. Since $E(\mathcal{P})\subseteq G-E(B)$, we have $E(B)\subseteq G-E(\mathcal{P})$.

\vspace{0.1cm}
\noindent \textbf{Claim 4.}
    $G-E(\mathcal{P})$ is $k$-connected.

\vspace{0.1cm}
\noindent \textbf{Proof of Claim 4.} Let $G^*=G-E(\mathcal{P})$. Since $E(B)\subseteq G^*$, $G^*$ contains $k$ pairwise vertex-disjoint paths $Q_1, Q_2,\ldots, Q_k$. For $i=1,2$, let $G_i^*=G^*[V(G_i)]$. We claim that $G_i^*$ is $k$-connected. For any $x\in V(G_i)$, since $G_i$ is a component of $G'-W=(G-E(\mathcal{P}_0))-W$, the edges incident with $x$ to other component in $G$ come only from $\mathcal{P}_0$. Thus,
$$|N_G(x)\cap V(G_i)|\ge \delta(G)-t-k'.$$
After removing $E(\mathcal{P})$, we have
$$d_{G_i^*}(x)\ge \delta(G)-t-k'-t\ge \delta(G) -2t-k+1.$$
On other hand, since $|V(G_i)|\le n-\delta(G)+t-1$, it is enough to verify 
$$\delta(G)-2t-k+1\ge \frac{n-\delta(G)+t+k-3}{2},$$
which is equivalent to verify
$$3\delta(G)\ge n+5t+3k-5.$$
If $t=3$, the desired inequality becomes $3\delta(G)\ge n+3k+10$, and it follows from $n\ge 6k+1$. If $t\ge 4$, then $\delta(G)\ge \lceil\frac{n+t+2}{2}\rceil$, and it still follows from $n\ge 6k+6t-17$.
Thus, we get 
$$\delta(G^*[V(G_i)])\ge \frac{|V(G_i)|+k-2}{2}.$$
Therefore, by Lemma~\ref{lem:CH}, $G_i^*$ is $k$-connected for $i=1,2$.

Consider the subgraph $F=G^*[V(G_1)\cup S \cup V(G_2)]$. We know that $F$ contains the two $k$-connected graphs $G^*[V(G_1)]$ and $G^*[V(G_2)]$, together with the $k$ pairwise internally disjoint paths $Q_1, \ldots, Q_k$ connecting $V(G_1)$ and $V(G_2)$. For any $w\in S$, we have
$$
|N_G(w)\cap (V(G_1)\cup V(G_2))| \ge \delta(G)-(|W|-1) \ge \delta(G)-k'+1.
$$
After deleting $E(\mathcal{P})$, we have
$$
|N_{G^*}(w)\cap (V(G_1)\cup V(G_2))| \ge \delta(G)-k'+1-t \ge \delta(G)-k-t+2.
$$
If $t=3$, then
$$\delta(G)-k-t+2\ge 2\lceil \frac{n+6}{2}\rceil -k-1\ge 2k+3\ge k.$$
If $t\ge 4$, then
$$\delta(G)-k-t+2\ge 2\lceil \frac{n+t+2}{2}\rceil -k-t+2\ge 2k+3t-5\ge k.$$
Hence, each $w\in S$ has at least $k$ neighbors belonging to $V(G_1)\cup V(G_2)$ in $G^*$. It follows from Lemma~\ref{lem:two-block} that $F$ is $k$-connected.

Let $W_0=W\setminus S$. If $W_0=\emptyset$, then we are done. Let $w\in W_0$, the same estimate as above gives
$$|N_{G^*}(w)\cap (V(G_1)\cup V(G_2))|\ge k.$$
Since $V(G_1)\cup V(G_2)\subseteq V(F)$, the vertex $w$ has at least $k$ neighbors in $F$. By Lemma \ref{lem:add vertex}, $G^*[V(F)\cup \{w\}]$ is $k$-connected. Repeat this process for all vertices in $W_0$. Then we conclude that $G^*[V(F)\cup W_0]$ is $k$-connected. Therefore, the graph $G^*$ is $k$-connected. 
That is, $\kappa(G-E(\mathcal{P}))\ge k$. The theorem follows.
\end{proof}

\section{Concluding Remarks}\label{sec5}

We have proved two connectivity-preserving results for spanning path structures. The first shows that, in sufficiently large $k$-connected graphs, the degree condition $\delta(G)\ge \lceil\frac{n+1}{2}\rceil$
guarantees a Hamiltonian path connecting any pair of vertices whose edge deletion preserves $k$-connectivity. The second conclusion gives a corresponding result for internally disjoint spanning $(u,v)$-paths system under a minimum degree condition. It is natural to ask whether the lower bound of the minimum degree condition can be reduced. Thus, we raise the following question.

\begin{problem}
    In Theorem~\ref{thm:path-system}, can the minimum degree condition be lowered to $\delta(G)\ge \lceil\frac{n+t}{2}\rceil$ for all $t\ge 3$?
\end{problem}

\section*{Declaration of competing interest}
 The authors declare that they have no known competing financial interests or personal relationships that could have appeared to influence the work reported in this paper.

\end{document}